\documentclass[10pt,reqno]{amsart}
\usepackage{amsmath, amsopn, amsfonts, amsthm, amssymb, amscd, enumerate, multicol, scalefnt}

\usepackage{amsmath, amsfonts, amssymb, amscd, enumerate, scalefnt}

\usepackage{etex}
\usepackage{hyperref}
\usepackage{t1enc}
\usepackage{graphicx}
\usepackage[all]{xy}
\usepackage{color}
\usepackage{slashed}
\usepackage[mathscr]{euscript}

\theoremstyle{plain}
\newtheorem{theorem}{Theorem}[section]
\newtheorem{definition}[theorem]{Definition}
\newtheorem{lemma}[theorem]{Lemma}
\newtheorem{proposition}[theorem]{Proposition}
\newtheorem{corollary}[theorem]{Corollary}
\newtheorem{remark}[theorem]{Remark}

\newtheorem{thmint}{Theorem}
\newtheorem{corint}[thmint]{Corollary}

\renewcommand{\b}{\begin{equation}}
\newcommand{\e}{\end{equation}}

\newcommand{\g}{\mathfrak{g}}
\newcommand{\h}{\mathfrak{h}}
\newcommand{\n}{\mathfrak{n}}
\newcommand{\p}{\mathfrak{p}}
\newcommand{\z}{\mathfrak{z}}

\renewcommand{\l}{\mathfrak{l}}

\newcommand{\la}{\langle}
\newcommand{\ra}{\rangle}
\newcommand\C{{\mathbb C}}
\newcommand\R{{\mathbb R}}

\renewcommand\H{{\mathbb H}}

\newcommand{\tr}{\operatorname{tr}}
\newcommand{\id}{\operatorname{Id}}

\newcommand{\kk}{\operatorname{K}}

\newcommand{\ip}{\langle\cdot,\cdot\rangle}
\newcommand{\HCFu}{\text{HCF}_+}
\title[Positive HCF on complex 2-step nilpotent Lie groups]{Positive Hermitian curvature flow on complex 2-step nilpotent Lie groups}
\subjclass[2010]{Primary 53C44; Secondary 53C15, 53C07, 53B15}
\thanks{This work was supported by G.N.S.A.G.A. of I.N.d.A.M} 

\author{Mattia Pujia}
\address{Dipartimento di Matematica G. Peano, Universit\`a di Torino, Via Carlo Alberto 10, 10123 Torino, Italy}
\email{mattia.pujia@unito.it}

\date{\today}

\begin{document} 

\begin{abstract} We study the positive Hermitian curvature flow of left-invariant metrics on complex 2-step nilpotent Lie groups. In this setting we completely characterize the long-time behaviour of the flow, showing that normalized solutions to the flow subconverge to a non-flat algebraic soliton, in Cheeger-Gromov topology. We also exhibit a uniqueness result for algebraic solitons on such Lie groups.
\end{abstract}

\maketitle
\section*{Introduction}
In 2011, Streets and Tian introduced a new family of parabolic flows generalizing the K\"ahler-Ricci flow to the Hermitian setting \cite{ST11}. More precisely, given a complex manifold $X$, any flow in the family evolves a Hermitian metric $g$ on $X$ via
\begin{equation}\label{HCFS}
\partial_t\, g_t= -S(g_t)+Q(g_t)\,,\qquad {g_t}_{|_0}=g\,,
\end{equation}
where $S(g)$ is the second Chern-Ricci curvature tensor of $g$ on $X$ and $Q(g)$ is a $(1,1)$-symmetric tensor quadratic in the torsion of the Chern connection. The flows belonging to such a family are usually called {\em Hermitian curvature flows} (HCFs for short).

In \cite{ST11}, Streets and Tian chose the tensor $Q$ in order to obtain a {\em gradient flow}, stable near K\"ahler-Einstein metrics with non-positive scalar curvature and satisfying many other analytical properties; while, in the subsequent paper \cite{ST10}, $Q$ was chosen so that the {\em pluriclosed condition} was preserved ($\partial \bar \partial \omega=0$) (see also \cite{S16, S18, ST12, ST13}). On the other hand, since the tensor $Q$ does not affect the parabolicity of the evolution equation \eqref{HCFS}, different choices of $Q$ can be performed in order to preserve other properties.\medskip

In \cite{U19}, following Streets and Tian's approach, Ustinovskiy introduced a HCF preserving both the Griffiths-positivity and the dual Nakano-positivity of the tangent bundle. More precisely, given a compact Hermitian manifold $(X,g)$, Ustinovskiy considered the evolution equation% on a compact Hermitian manifold $(X,g_0)$
\begin{equation}\label{HCF}
\partial_t g_t= -S(g_t)- \widetilde Q(g_t)\,, \qquad {g_t}_{|_0}=g\,.
\end{equation}
Here, we denote by: $\nabla$ the Chern connection of $(X,g)$, $\Omega$ the curvature tensor of $\nabla$, $S(g)$ the $(1,1)$-symmetric tensor 
$$
S_{i\bar j}=g^{k\bar l}\Omega_{k\bar l i\bar j}\,,$$
and $\widetilde Q(g)$ the tensor given by
$$
2\, \widetilde Q_{i\bar j}= g^{k\bar l}g^{m\bar n}T_{km \bar j}T_{\bar l \bar n i}\,,
$$
where $T_{km \bar j}:=g_{l\bar j}T_{km}^l$ and $T_{km}^l$ are the components of the torsion of $\nabla$. We will refer to \eqref{HCF} as to the {\em positive Hermitian curvature flow} ($\HCFu$ for short) and we set
$$
K(g):=S(g)+\widetilde Q(g)\,,
$$
for any Hermitian metric $g$ on $X$.\medskip

The aim of the present paper is to study the behaviour of the $\HCFu$ on complex nilpotent Lie groups when the initial metrics are left-invariant. Although in the non-compact case existence and uniqueness of left-invariant solutions to \eqref{HCF} are not always guaranteed, in our setting the invariance by biholomorphisms of the flow implies that \eqref{HCF} can be reduced to an ODE on the Lie algebra of the group. Hence, existence and uniqueness among left-invariant solutions follow from the standard ODE theory.\medskip% Motivated by this, in the following any solutions to the $\HCFu$ has to be intended as the unique left-invariant solution to the flow.\medskip

Our first result completely describes the long-time behaviour of the $\HCFu$ on a complex 2-step nilpotent Lie group.
\begin{thmint}\label{main}
Any left-invariant solution $g_t$ to the $\HCFu$ on a complex 2-step nilpotent Lie group is immortal, and $(1 + t)^{-1}g_t$ subconverges as $t\to \infty$ to a non-flat algebraic soliton $(\bar N, \bar g)$, in the Cheeger-Gromov topology.  
\end{thmint}

By convergence in the Cheeger-Gromov topology we mean that: there exists a family of biholomorphisms $\varphi_t : \Omega_t \subset \bar N \to \varphi_t(\Omega_t) \subset N$ mapping the identity of $\bar N$ into the identity of $N$, such that the open sets $\{\Omega_t\}$ exhaust $\bar N$, and in addition $\varphi_t^* g_{t} \to \bar g$ as $t\to\infty$, uniformly over compact subsets in the $C^\infty$-topology.\medskip %Remarkably, even if the space $\bar N$ might not be diffeomorphic to $N$, it still remains a complex 2-step nilpotent Lie group.

Next we focus on soliton solutions to the flow. A Hermitian metric $g$ on a complex Lie group $G$ is a {\em soliton to the $\HCFu$} if 
\begin{equation}\label{soliton}
K(g) = c\, g +\mathcal{L}_Z g\,,
\end{equation}
for some $c\in \R$ and a complete holomorphic vector field $Z$. Here $\mathcal{L}$ denotes the Lie derivative. Notice that, soliton metrics are rather important in the study of the $\HCFu$. Indeed, by the scale and biholomorphisms invariance of the $\HCFu$-tensor, any initial metric $g$ satisfying \eqref{soliton} gives rise to a self-similar solution 
$$
g_t=s(t)\varphi_t^\ast(g)
$$ 
to the $\HCFu$, for some smooth scaling function $s(t)>0$ and a one-parameter family of biholomorphisms $\varphi_t:G\to G$. If furthermore $g$ is left-invariant and $\varphi_t$ may be chosen to be a family of Lie group automorphisms, then the soliton $g$ is said to be {\em algebraic}.\medskip

Our second result is about the uniqueness of algebraic solitons on complex 2-step nilpotent Lie groups. 

\begin{thmint}\label{main_sol} Any complex 2-step nilpotent Lie group $N$ admits at most one algebraic $\HCFu$-soliton up to homotheties. Moreover, any algebraic $\HCFu$-soliton on $N$ is expanding (i.e. $c<0$ in \eqref{soliton}).
\end{thmint}

The proofs of our results are mainly based on the {\em bracket flow technique}, a powerful tool introduced by Lauret to study different geometric flows on homogeneous spaces (see e.g. \cite{LauGen}), and on a {\em moment map result} which we prove in Section \ref{sec_long}.\medskip

We mention that recently many works on different HCFs on homogeneous spaces appeared. In \cite{U18} Ustinovskiy studied the $\HCFu$ on complex homogeneous manifold focusing on a distinguished class of non-homogeneous metrics, namely the class of {\em induced metrics}; while, %, and by Ustinovskiy's work it also follows that any left-invariant solution to the $\HCFu$ on complex solvable Lie groups is immortal. 
the $\HCFu$ on {\em C-spaces} has been investigated by Panelli and Podest\`a in \cite{PodPan19}. Moreover, in \cite{Stan20} Stanfield studied the $\HCFu$ on almost-abelian complex Lie groups.

The general behaviour of the `original' HCF on {\em complex unimodular Lie groups} was studied by Lafuente, Vezzoni and the author in \cite{LPV}, where it is proved that any left-invariant solution to the flow is immortal and converges to an expanding algebraic soliton, which has to be unique up to homotheties. Moreover, in \cite{P} it has been shown that expanding algebraic solitons on complex Lie groups lead to strong constrains on the algebraic structure. Finally, the behaviour of the `original' HCF of locally homogeneous non-K\"ahler metrics on {\em compact complex surfaces}, together with a convergence result, was investigated in \cite{PP}.

Also the {\em pluriclosed flow} has been widely studied in the homogeneous case. In \cite{EFV15} it was proved that the flow of left-invariant metrics on 2-step nilpotent Lie groups is immortal (see also \cite{PV}); while, in \cite{AL} convergence results for normalized solutions to the flow on almost abelian and 2-step nilpotent Lie groups were proved.\medskip

In \cite{FP19_2}, Fei and Phong showed that there exists a relation between the $\HCFu$ and the Anomaly flow, which is a metric flow introduced by Phong, Picard and Zhang in \cite{PPZ18} to study the Hull-Strominger system (see also \cite{PPZ18_2,PPZ18_3,PPZ19_2}). Namely,  Fei and Phong proved that any solution to the $\HCFu$ starting from a {\em conformally balanced} metric $\omega$ on $X$, i.e. a metric satisfying $d(\lVert\Psi\rVert^2_\omega\, \omega^{n-1})=0$ for a complex volume form $\Psi$ on $X$ with $\dim_\C X=n$, gives rise to a solution to the Anomaly flow after a time rescaling. Thus, since any left-invariant Hermitian metric on a complex unimodular Lie group is balanced \cite{AG86}, our results apply to the Anomaly flow when $\Psi$ is a left-invariant volume form.

\begin{corint} Let $N$ be a complex 2-step nilpotent Lie group and $\Psi$ a left-invariant volume form on $N$. Any left-invariant solution $\omega_t$ on $N$ to the Anomaly flow 
$$
\partial_t(\lVert \Psi\rVert_{\omega_t}\, \omega_t^{n-1}) = i\,\partial \bar\partial \omega_t^{n-2}
$$
is immortal. Moreover, $(1 + t)^{-1}\omega_t$ subconverges as $t\to \infty$ to a non-flat left-invariant metric $\bar \omega$ on $\bar N$ satisfying \eqref{soliton}, in the Cheeger-Gromov topology.  
\end{corint}

In \cite{PPZ19} Phong, Picard and Zhang studied the Anomaly flow on complex 3-dimensional unimodular Lie groups. We mention that our corollary slightly extend the long-time existence result obtained in \cite{PPZ19} for the complex Heisenberg Lie group. Finally, in \cite{PU} the author and Ugarte investigated the behaviour of the Anomaly flow on 2-step nilpotent Lie groups providing the first examples of the Anomaly flow with non-flat holomorphic bundle.\medskip

The paper is organized as follows. In Section \ref{metric_ev} we explicitly compute the components of the $\HCFu$-tensor in terms of the Lie bracket. We also recall the bracket flow technique. Section \ref{sec_long} is devoted to the proof of Theorem \ref{main}. Finally, in Section \ref{sec_sol} we prove Theorem \ref{main_sol} and we give an explicit example.\medskip 

\noindent {\bf Notation and conventions.} All over the paper we adopt the Einstein convention for the sum over repeated indices. By a {\em complex Lie group} we mean a Lie group endowed with a bi-invariant complex structure (that is, the multiplication is a holomorphic map).\smallskip

\noindent {\bf Acknowledgments.} The author warmly thanks Luigi Vezzoni and Ramiro Lafuente for their interest and helpful comments.

\section{Preliminaries}\label{metric_ev}
\subsection{The $\HCFu$ on Lie groups}
Let $(G,J)$ be a Lie group equipped with a left-invariant complex structure. Let $\g$ be the Lie algebra of $G$ and let $\mu$ be the Lie bracket on $\g$. In the following, we compute the tensor $K$ of a left-invariant Hermitian metric $g$ on $G$ in terms of structure constants of $\g$.\smallskip

The Chern connection $\nabla$ of $g$ is by definition the unique Hermitian connection ($\nabla g=\nabla J=0$) with vanishing (1,1)-part of the torsion. Therefore, given a left-invariant $g$-unitary frame $\{Z_1,\ldots,Z_n\}$ on $G$, we have
$$
T_{l\bar k}:=\nabla_l Z_{\bar k}-\nabla_{\bar k}Z_l-\mu(Z_l,Z_{\bar k})=0\,,
$$
or, in terms of the Christoffel symbols of $\nabla$,
$$
\Gamma_{l\bar k}^{\bar r}=\mu_{ l \bar k}^{\bar r}\,,\quad \Gamma_{\bar kl}^r=\mu_{\bar kl}^r\,.
$$
On the other hand, by $\nabla g=\nabla J=0$ we get
$$
g(\nabla_{l}Z_r,Z_{\bar j})=-g(Z_r,\nabla_{l}Z_{\bar j})=-g(Z_r,\mu(Z_l,Z_{\bar j}))=-\mu_{l\bar j}^{\bar r}\,,
$$
that is
\begin{equation}\label{Christ_sym}
\Gamma_{lr}^{j}=-\mu_{l\bar j}^{\bar r} \,.
\end{equation}
Therefore, we have
$$
\Omega_{k\bar l i\bar j}= g(\nabla_{k}\nabla_{\bar l}Z_i,Z_{\bar j})-g(\nabla_{\bar l}\nabla_{k}Z_i,Z_{\bar j})-g(\nabla_{\mu(Z_k,Z_{\bar l})}Z_i,Z_{\bar j})
$$
and hence
$$ 
S_{i\bar j}=-\mu_{\bar ki}^r\mu_{k\bar j}^{\bar r}+\mu_{k\bar r}^{\bar i} \mu_{\bar k r}^j+\mu_{k\bar k}^r\mu_{r\bar j}^{\bar i}-\mu_{k\bar k}^{\bar r}\mu_{\bar ri}^{j}\,.
$$
Finally, since 
$$
T_{ip}:=\nabla_{i}Z_p-\nabla_pZ_i-\mu(Z_i,Z_p)\,,
$$
%we have 
%$$
%T_{ip}^k:=\Gamma_{ip}^k-\Gamma_{pi}^k-\mu_{ip}^k
%$$
%and
by means of \eqref{Christ_sym} we have
$$
T_{ip}^k:=-\mu_{i\bar k}^{\bar p}+\mu_{p\bar k}^{\bar i}-\mu_{ip}^k\,,\quad 
T_{ip\bar r}=-\mu_{i\bar r}^{\bar p}+\mu_{p\bar r}^{\bar i}-\mu_{ip}^r\,,
$$
which implies
$$
2\,\widetilde Q_{i\bar j}=T_{\bar k\bar r i}T_{kr \bar j}=\left(-\mu_{\bar ki}^{r}+\mu_{\bar ri}^{k}-\mu_{\bar k\bar r}^{\bar i}\right)\left(-\mu_{k\bar j}^{\bar r}+\mu_{r\bar j}^{\bar k}-\mu_{kr}^j\right)\,.
$$

As a direct consequence we have the following proposition.

\begin{proposition}\label{lemm_tens} Let $(G,g)$ be a complex Lie group equipped with a left-invariant Hermitian metric. Then, with respect to a left-invariant $g$-unitary frame $\{Z_1,\ldots,Z_n\}$ on $G$, we have
$$
K(g)(Z_i,Z_{\bar j})= \frac {1}{2}\, g(\mu(Z_{\bar r},Z_{\bar p}),Z_i)\cdot g(\mu(Z_r,Z_p),Z_{\bar j})\,,
$$ 
\end{proposition}

\begin{proof} Since on complex Lie groups any mixed bracket $\mu(Z_l, Z_{\bar s})$ vanishes, the claim follows by the above computations.
\end{proof}

\subsection{The bracket flow technique}\label{sub_brack}
Here we briefly recall the bracket flow technique introduced by Lauret in \cite{LauMatAnn} to study the Ricci flow on nilmanifolds. This technique provides a method to study a prescribed geometric flow via a flow of Lie brackets and it has been extensively used to investigate different metric flows in Hermitian geometry (see e.g. \cite{AL}, \cite{EFV15}, \cite{LPV}, \cite{LauHer}, \cite{PV}). Under some natural assumptions, the bracket flow technique also applies to a large class of geometric structures on homogeneous spaces \cite{LauGen}.\smallskip

Let $(G,J)$ be a complex $n$-dimensional Lie group and let $\g$ be its Lie algebra. Then, the Lie bracket $\mu_0$ of $\g$ belongs to the {\em variety of complex Lie algebras}
\[
\mathcal C=\left\{\mu\in\Lambda^2\g^*\otimes \g :\text{$\mu$ satisfies the Jacobi identity and }\mu(J\cdot,\cdot)=J\mu(\cdot,\cdot)\right\}\,,%\subseteq \Lambda^2\g^*\otimes \g\,,
\]
which admits the `natural' action of 
$$
{\rm Gl}(\g,J) = \left\lbrace \varphi\in{\rm Gl}(\g) : \varphi\circ J=J\circ\varphi \right\rbrace\simeq {\rm Gl}_n(\C)
$$ 
defined by
\begin{equation}\label{act_alg}
\varphi\cdot\mu:=\varphi\,\mu(\varphi^{-1}\cdot,\varphi^{-1}\cdot)\,.
\end{equation}

Now, let us consider a left-invariant Hermitian metric $g_0$ on $G$. By the biholomorphisms invariance of $K$, there exists a unique solution $(g_t)_{t\in I}$, $0\in I\subseteq \R$, to the $\HCFu$ \eqref{HCF} consisting entirely of left-invariant Hermitian metrics. Moreover, in \cite[Theorem 1.1]{LauHer} Lauret proved that there exists a smooth curve $(\varphi_t)_{t\in I}\in{\rm Gl}(\g,J)$, with $\varphi_0=\id_\g$, such that
$$
g_t(\cdot,\cdot)=g_0(\varphi_t\cdot,\varphi_t\cdot)\,,
$$
and the family of Lie brackets
$$
\mu_t:= \varphi_t\cdot\mu_0
$$
satisfies the {\em bracket flow equation}
\begin{equation}\label{br_flow}
\frac{d}{dt}\mu_t = -\pi(\kk_{\mu_t})\mu_t\,,\qquad {\mu_t}_{\vert_0}=\mu_0\,.
\end{equation}
Here, $\pi: {\rm End}(\g)\to {\rm End}(\Lambda^2\g^*\otimes \g)$ denotes the representation of the action defined in \eqref{act_alg}, i.e.
\begin{equation}\label{act_brac}
\pi(A)\mu (\cdot,\cdot)= A\mu(\cdot,\cdot)-\mu(A\cdot,\cdot)-\mu(\cdot,A\cdot)\,,\qquad A\in{\rm End}(\g)\,;
\end{equation}
while $\kk_{\mu_t}:\g\to\g$ is the endomorphism related to the value of $K(g_t)$ at $e\in G$ by
\begin{equation}\label{end_tens}
\kk_{\mu_t} = \varphi_t \kk_{g_t} \varphi_t^{-1}\,,\qquad g_t(\kk_{g_t}\cdot,\cdot)=K(g_t)(\cdot,\cdot)\,.
\end{equation}

%
%Finally, since the bracket flow preserves the metric, any tensor product of $\g$ and $\g^\ast$ can be endowed with an inner product induced by $g_0$. For instance, given $A,B\in {\rm End}(\g)$, then
%$$\langle A,B\rangle:=\tr A\bar B^t\,,$$
%where the adjoint $(\cdot)^t$ is given with respect to $g_0$. 
%
%To simplify the notation we will also denote by $\ip$ the inner product on $\Lambda^2\g^*\otimes \g$, and the difference will be clear by the context.

\section{The $\HCFu$ on complex 2-step nilpotent Lie groups}\label{sec_long}
This section is devoted to the proof of Theorem \ref{main}. The main feature in the proof will be a geometric invariant theory result, which implies the convergence claim.\medskip

Let $(N,J)$ be a complex 2-step nilpotent Lie group equipped with a left-invariant Hermitian metric $g$. Let $\n$ be the Lie algebra of $N$ and $\mu_0$ its Lie bracket. Let also $\z$ be the center of $(\n,\mu_0)$ and $\z^\perp$ its $g$-orthogonal complement, that is $\n=\z^\perp\oplus\z$. By means of Proposition \ref{lemm_tens}, the endomorphism $\kk_g$ defined in \eqref{end_tens} satisfies
$$
\kk_g(X)=0\,,\quad\text{for every}\,\, X\in\z^\perp\,,
$$
or, equivalently, 
\begin{equation}\label{K_oper}
\kk_g=\begin{bmatrix} 0 & 0\\ 0 & \ast \end{bmatrix}\,.
\end{equation}
with respect to the block representation $\n=\z^\perp\oplus\z$. As a direct consequence, we get

\begin{proposition}\label{prop_ort} The $\HCFu$ starting from a left-invariant Hermitian metrics on $N$ preserves the splitting $\g=\z^\perp\oplus\z$.
%$$
%\frac{d}{dt}g_t(X,\cdot)=0\,,\quad\text{for every}\,\,X\in\z^\perp\,.
%$$
\end{proposition}

\begin{remark}\rm A generalization of Proposition \ref{prop_ort} holds for complex $k$-step nilpotent Lie group, with $k\geq2$. More precisely, let $(H,g)$ be a complex nilpotent Lie group equipped with a left-invariant Hermitian metric. Let also $\h$ be the Lie algebra of $H$, $\mu$ its Lie bracket and $\h'=\mu(\h,\h)$ its commutator. Then, $\kk_g(X)=0$ holds for every $X\in(\h')^\perp$, where $(\h')^\perp$ is the $g$-orthogonal complement of $\h'$.
\end{remark}

Now, let us consider the {\em variety of complex 2-step nilpotent Lie algebras} 
$$
\mathcal N:=\{\mu\in\mathcal C: \mu(\mu(\cdot,\cdot),\cdot)=0 \}\,,
$$ 
and the group
$$
{\rm Gl}(\z,J):=\{\tilde \varphi\in{\rm Gl}(\z): \tilde \varphi\circ J=J\circ\tilde \varphi\}\subset {\rm Gl}(\n,J)\,,
$$
which can be considered as a subgroup of ${\rm Gl}(\n,J)$ via the embedding map $\tilde \varphi \mapsto\scriptstyle \left(\begin{smallmatrix}\id &0 \\ 0 & \tilde \varphi\end{smallmatrix}\right)$. Moreover, let $\g\l(\z,J)$ be the Lie algebra of ${\rm Gl}(\z,J)$ and 
$$
\p(\z,J):=\g\l(\z,J)\cap \p\,,
$$
where $\p={\rm sym}(\n,\ip)$ is the set of symmetric endomorphisms of $\n$. Here $\ip$ denotes any inner product induced by $g$ on any tensor product of $\n$ and its dual. Then, we have the following fundamental lemma.

\begin{lemma}\label{lemm_GIT} The mapping 
\begin{equation}\label{mom_map}
\mathcal N\ \backslash \{ 0\} \xrightarrow{\Phi} \p(\z,J)\,, \qquad\mu \mapsto \frac{2}{\Vert \mu \Vert^2} \, \kk_\mu\,,
\end{equation}
is a moment map for the linear ${\rm Gl}(\z,J)$-action on $\mathcal N\ \backslash \{0\}$, in the sense of real geometric invariant theory. That is, 
\begin{equation}\label{eqn_mm}
	\la \kk_\mu, E \ra = \frac12 \, \la  \pi(E)\mu, \mu \ra, \qquad \mbox{for all }\,\, E\in \p(\z,J)\,, \,\,\, \mu \in \mathcal N\ \backslash \{ 0\}\,,
\end{equation}
where $\pi(E)\mu$ is defined as in \eqref{act_brac}. 
\end{lemma}

\begin{proof}
%The claim follows by a straightforward computation. 
Let $\{Z_i\}$ be a $\ip$-unitary basis of $\n$. By means of the 2-step nilpotency, one gets that
$$
\la \kk_\mu(Z_i), E(Z_{\bar i}) \ra = \frac 12 \mu_{kr}^i\mu_{\bar k \bar r}^{\bar j} E_{\bar i}^{\bar j}\,,\qquad \text{for any }  E\in \p(\z,J)\,.
$$
On the other hand,
$$
\la  \pi(E)\mu(Z_r,Z_k), \mu(Z_{\bar r},Z_{\bar k}) \ra = \la E\mu(Z_{ r},Z_{ k})\,,\mu(Z_{\bar r},Z_{\bar k})\ra= E_{i}^s\mu_{rk}^s\mu_{\bar r \bar k}^{\bar i}\,,
$$
and the claim follows by the symmetry of $E$.
\end{proof}

Let us stress that, a similar result have been found for the `original' HCF on complex unimodular Lie groups, and $\Phi$ is actually the restriction (up to a positive constant) of that moment map found in \cite{LPV}. Indeed, if we denote by
$$
\mathcal N\ \backslash \{ 0\} \rightarrow {\rm End}(\n)\,, \qquad\mu \mapsto \frac{4}{\Vert \mu \Vert^2} \, {\rm M}_\mu\,,
$$
the moment map associated to the `original' HCF, which satisfies
$$
	\la {\rm M}_\mu, A \ra = \frac12 \, \la  \pi(A)\mu, \mu \ra, \qquad \text{for all }\,\, A\in {\rm End}(\n)\,, \,\,\, \mu \in \mathcal N\ \backslash \{ 0\}\,,
$$
(see \cite[Sec. 3]{LPV}), a direct computation yields $\la {\rm M}_\mu,  E \ra=2\,\la \kk_\mu,  E \ra$ for every $ E\in\p(\z,J)$.

\begin{remark}\rm By means of \cite[Lemma 2.1]{LPV}, it follows that the endomorphism $\kk_\mu:\n\to\n$ is given by
$$
\kk_\mu= \begin{bmatrix} 0 & 0\\ 0 &  2\cdot{\rm pr_\z}({\rm Ric_\mu}) \end{bmatrix}\,,
$$
where ${\rm Ric_\mu}$ is the Ricci curvature endomorphism of the Ricci curvature tensor ${\rm Ric}(g)$ and ${\rm pr_\z}$ denotes its orthogonal projection to ${\rm End}(\z)$.
\end{remark}

We are now in a position to prove Theorem \ref{main}.

\begin{proof}[Proof of Theorem \ref{main}]% The proof of this theorem traces the one Theorem \ref{main_cxHCF}.\medskip

Let $g_t$ be the solution to the $\HCFu$ starting at $g_0$. By the equivalence of the $\HCFu$ and the bracket flow \eqref{br_flow}, it is enough to prove that $\mu_t$ is defined for all $t\in[0,+\infty)$. Therefore, since Lemma \ref{lemm_GIT} holds, we have
$$
\tfrac d{dt}\lVert\mu_t\rVert^2= 2\langle \tfrac d{dt}\mu_t,\mu_t\rangle= -2\langle \pi({\kk_{\mu_t}})\mu_t,\mu_t\rangle=-4  \lVert \kk_{\mu_t}\rVert^2\leq0\,,
$$
and the claim follows by standard ODE arguments.

Now, let $\nu_t:=\mu_t/\lVert\mu_t\rVert$ be the norm-normalized bracket flow. Then, by \cite[Lemma 2.3]{AL}, $\nu_t$ solves the {\em normalized bracket flow equation}
\begin{equation}\label{norm_brack}
\tfrac d{dt} \nu_t=-\pi(U_{\nu_t}+r_{\nu_t}\, \id_\n)\nu_t\,,
\end{equation}
where $r_\nu:=\langle \pi(\kk_{\nu})\nu,\nu\rangle=2 \lVert \kk_\nu\rVert^2$. On the other hand, since \eqref{mom_map} is a moment map, by means of \cite[Lemma 7.2]{BL} the normalized bracket flow turns out to be the negative gradient flow (up to a constant and a time reparameterization) of the real-analytic functional
$$
F:\mathcal N\ \backslash \{0\} \to \R\,,\quad \nu\mapsto\frac{\lVert \kk_\nu\rVert^2}{\lVert\nu\rVert^4}\,.$$
Moreover, since $\nu_t$ exists for all $t\geq0$ and the space of unitary bracket is compact, there must exist an accumulation point $\bar \nu$ of $\nu_t$. Then, by {\L}ojasiewicz's theorem on real-analytical gradient flow, $\nu_t\to\bar \nu$ as $t\to\infty$ and 
$$\pi(\kk_{\bar \nu}+r_{\bar \nu}\, \id_\n)\bar\nu=0\,,$$
i.e. $\bar\nu$ is a fixed point of \eqref{norm_brack}. Therefore, in view of \eqref{act_brac}, $\kk_{\bar \nu}+r_{\bar \nu}\, \id_\n$ is a derivation of $\n$ and its corresponding metric $\bar g$ is an algebraic $\HCFu$-soliton. A direct computation yields that 
$$ 
{\rm sc} (\bar g)=\tr \kk_{\bar \nu}=-\frac 12\,,
$$
and hence the soliton is non-flat.

Finally, arguing in the same fashion as \cite[Theorem A]{AL}, it is not hard to prove that $\lVert \mu(t)\rVert\sim t^{-1/2}$ as $t\to\infty$. On the other hand, scaling the metric by a factor ${c>0}$ is equivalent to scaling the corresponding bracket by a $c^{-1/2}$ factor (see \cite[$\S$2.1]{LauMZ}). Therefore, since the convergence of the brackets yields subconvergence for the corresponding family of left-invariant metrics in the Cheeger-Gromov sense \cite[Theorem 6.20]{LauLMS}, the claim follows.
\end{proof}

\section{Algebraic solitons to the $\HCFu$}\label{sec_sol}

Let $(G,g)$ be a complex non-abelian Lie group equipped with a left-invariant metric. Then, $g$ is said to be a {\em static metric} to the $\HCFu$ if
\begin{equation}\label{static}
K(g)=c\, g\,,
\end{equation}
for some $c\in \R$.
\begin{proposition}\label{prop_stat} Any left-invariant $\HCFu$-static metric $g$ on $G$ is shrinking (i.e. $c>0$).
\end{proposition}

\begin{proof} The proof directly follows by Proposition \ref{lemm_tens}, since $n\,c={\tr_g K(g)}=\frac 12 \lVert\mu\rVert^2$.
\end{proof}

As a direct consequence, we get the following

\begin{corollary}\label{cor_stat} If $G$ is nilpotent, then there are no left-invariant $\HCFu$-static metrics on $G$.
\end{corollary}

\begin{proof} Let us suppose that there exists a left-invariant metric $g$ on $G$ satisfying \eqref{static}. Let $\{Z_1,\ldots,Z_n\}$ be a $\ip$-unitary basis of $\g$. Since $G$ is nilpotent, there exists at least one $Z_i$ such that $\la Z_i\ra\perp \mu(\g,\g)$. By means of Proposition \ref{lemm_tens}, one gets
$$
K(g)(Z_i,Z_{\bar i})=0\,,
$$
which is not possible by Proposition \ref{prop_stat}, and hence the claim follows.
\end{proof}

\begin{remark}\rm Given a complex Lie algebra $(\g,J)$, the complex structure $J$ preserves both the center $\z$ and the commutator $\mu(\g,\g)$ of $\g$. This was implicitly used in the proof of Corollary \ref{cor_stat}.
\end{remark}

Let us now focus on soliton metrics.

\begin{definition} A left-invariant metric $g$ on $G$ is said to be an {\em algebraic $\HCFu$-soliton} if 
$$
K(g)=c\, g+ \frac12\,\big( g(D\cdot,\cdot)+g(\cdot,D\cdot)\big)\,,%\quad \text{for some}\quad c\in\R\,,\quad D\in{\rm Der}(\g)\,.
$$
for some $c\in \R$ and $D\in{\rm Der}(\g)$.
\end{definition}
\smallskip

%\begin{proposition} Let $N$ be a complex 2-step nilpotent nilmanifold. Every algebraic $\HCFu$-soliton $g$ on $N$ is expanding. Furthermore, it satisfies $D^t=D$.
%\end{proposition}

We are now in a position to prove Theorem \ref{main_sol}.

\begin{proof}[Proof of Theorem \ref{main_sol}] Let $(N,g)$ be a complex (non-abelian) 2-step nilpotent Lie group equipped with an algebraic $\HCFu$-soliton metric. Moreover, let $\n$ be the Lie algebra of $N$, $\mu$ its Lie bracket and let us consider the $\ip$-orthogonal splitting $\n=\z^\perp\oplus\z$, where $\z$ is center of $\n$. Then, by means of \eqref{end_tens}, the algebraic soliton condition can be written in terms of $\kk_\mu$ as
\begin{equation}\label{sol_oper}
\kk_\mu=c\, \id_\n + \frac12\, \big(D+D^t\big)\,,
\end{equation}
where $(\cdot)^t$ denotes the transpose with respect to the inner product $\ip$ on $\n$. \smallskip

We first show that any derivation $D$ in \eqref{sol_oper} is a symmetric operator, for which its enough to prove that $D^t$ is also a derivation. Since $\n$ is a complex 2-step nilpotent Lie algebra, it follows that
$$
D=\begin{bmatrix} \ast &0 \\\ast& \ast\end{bmatrix}
$$
where the blocks are in terms of $\n=\z^\perp\oplus\z$. On the other hand, by means of \eqref{K_oper} and \eqref{sol_oper}, we have
\begin{equation}\label{D_matrix}
D=\begin{bmatrix} -c\,\id_{\z^\perp}&0\\0& D_{\z}\end{bmatrix}\,.
\end{equation}
Therefore, since the representation $\pi$ is a Lie algebra morphism which satisfies $\pi(A^t)=\pi(A)^t$ and Lemma \ref{lemm_GIT} holds, we get
$$
2 \tr \kk_\mu [D,D^t]=\la \pi([D,D^t])\mu,\mu\ra=\la [\pi(D),\pi(D^t)]\mu,\mu\ra= \lVert\pi(D^t)\mu\rVert^2\,;
$$
while, equation \eqref{sol_oper} implies
$$
\tr \kk_\mu [D,D^t]= c\, \tr [D,D^t]+ \frac 12 \tr D[D,D^t]+ \frac 12 \tr D^t[D,D^t]=0\,,
$$
and the claim follows by $\lVert\pi(D^t)\mu\rVert^2=0$, i.e. $D^t\in{\rm Der}(\n)$. \smallskip

We are now in a position to prove that any algebraic $\HCFu$-soliton on $N$ is expanding (i.e. $c<0$). Let us focus on soliton metrics satisfying \eqref{sol_oper}. In view of Corollary \ref{cor_stat}, we already know that $D\neq 0$. Moreover, by means of Theorem \ref{main}, a soliton metric cannot be shrinking (i.e. $c>0$), since it would give rise to a self-similar solution to the $\HCFu$ with a finite-time singularity (see e.g. \cite{LauHer}). Thus, we get $c\leq0$ in \eqref{sol_oper}. On the other hand, if we assume $c=0$, it directly follows by \eqref{D_matrix} that 
\begin{equation}\label{contr_proof}
\tr D_\z^2 = \tr \kk_\mu D_\z = \la \pi(D_\z)\mu,\mu\ra=0\,,
\end{equation}
which is in contradiction with our hypothesis $D\neq 0$. Hence, every algebraic $\HCFu$-soliton on $N$ has to be expanding.\smallskip

To finish the proof, we now prove the uniqueness claim. By the proof of Theorem \ref{main}, the operator $\kk_\mu$ gives rise to an algebraic $\HCFu$-soliton if and only if it is a critical point of the functional $F(\mu)=\lVert\kk_\mu\rVert^2/\lVert\mu\rVert^4$. On the other hand, it has been proved in \cite[Corollary 9.4]{BL} that critical points for the norm of the moment map which lie in a fixed orbit ${\rm Gl}(\z,J)\cdot\mu$ must actually lie in the same $(\p(\z,J)\cdot \mu)$-orbit. Then, since two brackets in the same $\p(\z,J)$-orbit correspond to isometric left-invariant metrics on $N$, this conclude the proof.
\end{proof}

As a direct consequence of the proof of Theorem \ref{main_sol}, one gets the following corollary.

\begin{corollary} Let $g$ be an algebraic $\HCFu$-soliton on $N$. Then, the $\HCFu$-tensor satisfies
$$
K(g)=c\, g+ g(D\cdot,\cdot)\,,%\quad \text{for some}\quad c\in\R\,,\quad D\in{\rm Der}(\g)\,.
$$
for some $c<0$ and $D\in{\rm Der}(\n)$.
\end{corollary}

\subsection{An explicit example}

Let $\H_3(\C)$ be the complex 3-dimensional Heisenberg Lie group, which can be defined as the matrix group 
$$
\H_3(\C)=\left\lbrace \left(\begin{matrix} 1& z_1& z_2\\ 0&1&z_3\\0&0&1\end{matrix}\right):z_i\in\C\right\rbrace\,.
$$
This group is 2-step nilpotent and admits a left-invariant $(1,0)$-frame $\{Z_1,Z_2,Z_3\}$ such that
$$\mu(Z_1,Z_2)=Z_3\,.$$

\begin{proposition} \label{Heisen} Every left-invariant Hermitian metric on $\mathbb H_3(\C)$ is an expanding algebraic $\HCFu$-soliton.
\end{proposition}

\begin{proof} Let $g$ be a left-invariant Hermitian metric on $\mathbb H_3(\C)$. Then, there exists a $g$-unitary (1,0)-frame $\{W_1,W_2,W_3\}$ such that $\mu$ satisfies
$$\mu(W_1,W_2)= s W_3\,, \quad \text{for some}\,\, s\in\C\,.$$
With respect to this new frame, we have
$$K_{g}=\frac 12\left(\begin{matrix} 0 & 0 & 0\\ 0 & 0 & 0 \\ 0 & 0 & |s|^2\end{matrix}\right)\,.$$
If we set $D:= K_{g}-c\,I$, then
$$DW_1=D_{11} W_1\,,\quad DW_2=D_{22} W_2\,,\quad DW_3=D_{33} W_3\,,$$
for some $D_i^i\in \R$, and hence $D$ is a derivation if and only if
$$D\mu(W_1,W_2)-\mu(DW_1,W_2)-\mu(W_1,DW_2) = (D_{33}-D_{11}-D_{22})W_3=0\,.$$
Finally, setting
$$c= -\frac 12|s|^2$$
the claim follows.
\end{proof}

\bibliographystyle{abbrv}
\bibliography{Biblio.bib}

\end{document}